\documentclass[a4paper,11pt]{amsart}

\usepackage{amssymb}  %,amscd,amsthm,mathrsfs}
\usepackage{mathtools}      % improves amsmath
\usepackage{mathabx}        % makes many symbols (as $\leq$) more beautiful; must be loaded after mathtools 
\usepackage[bb=fourier,cal=euler,scr=rsfs]{mathalfa}	% selecting fancy math fonts
\usepackage{enumitem}       % special itemize with custom tags and labels that work 

\usepackage[utf8]{inputenc}

\newcommand{\C}{\mathcal{C}}

\newtheorem*{q}{\sc Questions}

\newtheorem*{teo*}{Theorem}

\newtheorem{teo}{Theorem}[section]

\newtheorem{cor}[teo]{Corollary}

\newtheorem{lema}[teo]{Lemma}
\newtheorem{prop}[teo]{Proposition}

\newcommand{\bi}{\begin{itemize}}
\newcommand{\ei}{\end{itemize}}

\theoremstyle{definition}

\theoremstyle{remark}
\newtheorem{obs}[teo]{Remark}

\newcommand{\B}{\mathcal{B}}
\newcommand{\F}{\mathcal{F}}
\newcommand{\U}{\mathcal{U}}
\newcommand{\W}{\mathcal{W}}
\newcommand{\R}{\mathcal{R}}

\newcommand{\vol}{\mathrm{vol}}

\usepackage[dvipsnames]{xcolor}

\usepackage{comment}

\DeclareMathOperator{\Hol}{Hol}
\DeclareMathOperator{\Id}{Id}
\DeclareMathOperator{\Ball}{B}
\DeclareMathOperator{\diam}{diam}
\DeclareMathOperator{\length}{length}

\newcommand{\ii}{W^{cu}(\mathcal{C})}

\begin{document}

\author[Verónica De Martino]{Verónica De Martino} 
\address{CMAT, Facultad de Ciencias, Universidad de la Rep\'ublica, Uruguay}
\email{vdemartino@cmat.edu.uy}

\author[Santiago Martinchich]{Santiago Martinchich} 
\address{CMAT, Facultad de Ciencias, Universidad de la Rep\'ublica, Uruguay}
\email{smartinchich@cmat.edu.uy}

\title[Partially hyperbolic with compact center foliation]{Codimension one compact center foliations are uniformly compact}
\thanks{This research was partially funded by CSIC (Iniciación) No. 133, CSIC Group No. 618 and FCV-2017-111}

\begin{abstract}
Let $f:M\to M$ be a dynamically coherent partially hyperbolic diffeomorphism whose center foliation has all its leaves compact. We prove that if the unstable bundle of $f$ is one-dimensional, then the volume of center leaves must be bounded in $M$.

%As consequences, we obtain:

%Combined with the work of Bohnet \cite{Boh}, modulo finite cover, $f$ fibers over an Anosov automorphism of the torus.

%Furthermore, for any stable or unstable dimension, together with the work of Gogolev \cite{Gog}, if $\dim M\leq 5$ then the volume of center leaves is uniformly bounded. In particular, the Sullivan example \cite{S} of a foliation by circles with unbounded lenght of leaves cannot be the center foliation of $f$.

%Let $f:M\to M$ be a codimension one, dynamically coherent partially hyperbolic diffeomorphism. In this paper we prove that if the center foliation has all its leaves compact, then the volume of center leaves must be bounded.
\end{abstract}

\maketitle

\section{Introduction}

\subsection{Context}
A diffeomorphism $f:M\to M$ in a closed manifold $M$ is said to be \emph{partially hyperbolic} if the tangent bundle $TM$ decomposes as a direct sum of continuous and $Df$-invariant subbundles $$TM=E^s\oplus E^c \oplus E^u$$ such that vectors in $E^s$ are uniformly contracted by $Df$, vectors in $E^u$ are uniformly contracted by $Df^{-1}$ and vectors in $E^c$ have an intermediate behavior. 

The bundles $E^s$ and $E^u$ uniquely integrate to $f$-invariant foliations $\W^u$ and $\W^s$, respectively (see e.g. \cite{HPS}). The bundles $E^s\oplus E^c$ and $E^c\oplus E^u$ may or may not integrate to foliations $\W^{cs}$ and $\W^{cu}$. If they do integrate to $f$-invariant foliations the diffeomorphism is said to be \emph{dynamically coherent}. 

If $f$ is dynamically coherent, the bundle $E^c$ also integrates to an invariant foliation $\W^c$ whose leaves are the connected components of the intersections of leaves of $\W^{cs}$ and $\W^{cu}$.

This work fits in the context of studying partially hyperbolic diffeomorphisms where $\W^c$ is a compact foliation (namely, all leaves of $\W^c$ are compact).

Since Sullivan presented his example in \cite{S} of a foliation by cicles with unbounded lenght of leaves, compact foliations have been divided in whether the volume of leaves is uniformly bounded or not (see section \ref{compactfoliations} for a more detailed discussion). In particular, in the uniformly bounded case the leaf space is Hausdorff and has a nice orbifold structure, while in the non uniformly bounded scenario the leaf space is not Hausdorff and may have a complicate structure (see e.g \cite{Eps} and \cite{Vog}). 

Pugh posed the following questions (see ~\cite{RHRHU}, ~\cite{Gog}).

\begin{q} 
\textit{Let $f:M\to M$ be a partially hyperbolic diffeomorphism with a compact center foliation $\W^c$. Is it true that the volume of center leaves is uniformly bounded? Is it true that $f$ can be finitely covered by a
partially hyperbolic diffeomorphism $\tilde f \colon \widetilde M \to \widetilde M$ so that
there is a fibration $p \colon \widetilde M \to N$ whose fibers are
the center leaves and an Anosov diffeomorphism $\bar f \colon N\to N$
such that $p$ is a semiconjugacy between $\tilde f$ and $\bar f$?}
\end{q}

Progress on this questions has been done by Bonatti-Wilkinson \cite{BoW}, Bohnet \cite{Boh}, Bohnet-Bonatti \cite{BoBo}, Carrasco \cite{Car} and Gogolev \cite{Gog}.

\subsection{Main result}

The main theorem of this work is the following:

\begin{teo}\label{main}
Let $f:M\to M$ be a dynamically coherent partially hyperbolic diffeomorphism with compact center foliation $\W^c$. If $\dim(E^u)=1$ then the volume of the center leaves is uniformly bounded. 
\end{teo}

In \cite{Boh} Bohnet has studied the case where the volume of the center leaves is uniformly bounded and $\dim(E^u)=1$. Combining her results with our main theorem yields:

\begin{cor}\label{cor1}
Let $f:M\to M$ be a dynamically coherent partially hyperbolic diffeomorphism with compact center foliation $\W^c$. If $\dim(E^u)=1$ then, modulo taking a double cover, the leaf space $M/\W^c$ is a torus $\mathbb{T}^d$ and the dynamics $F:M/\W^c\to M/\W^c$ induced by $f$ is topologically conjugate to an Anosov automorphism on $\mathbb{T}^d$, where $d=codim(\W^c)$.
\end{cor}

Observe that Theorem \ref{main} and Corollary \ref{cor1} are valid as well if $\dim(E^s)=1$ by working with $f^{-1}$ instead of $f$.

In \cite{Gog} Gogolev proved that compact center foliations are uniformly compact under the assumptions $\dim(E^c)=1$, $\dim(E^s)\leq 2$ and $\dim(E^u)\leq 2$. Combining this result with our main theorem then yields:

\begin{cor} Let $f:M\to M$ be a dynamically coherent partially hyperbolic diffeomorphism with compact center foliation. If $\dim(M)\leq 5$ then the volume of the center leaves is uniformly bounded. 
\end{cor}

In particular, with the second corollary we can completely discard to be a center foliation the Sullivan's example ~\cite{S} of a foliation by circles in a 5-manifold with unbounded length of leaves. Note that for this specific example the result is, as far as we are aware, new. In Remark \ref{NoSullivan} we give a direct proof for this example without the need of the results from section \ref{four}.

\subsection{Organization of the paper and structure of the proof}

In section \ref{preliminaries} we give some preliminaries from partially hyperbolic dynamics  and foliation theory. In particular, we address the topic of compact foliations and review some definitions and results that will be useful in the proof of the main theorem.

In sections \ref{Section3} and \ref{four} we give the proof of the main theorem, and the structure is as follows.

The proof will be by contradiction. Assume that the center foliation $\W^c$ has not uniformly bounded volume of leaves. This is equivalent with the \emph{bad set}
$$\B=\{x:\text{ center leaf volume function is not locally bounded at } x\}$$ being not empty (see section \ref{compactfoliations}).

The main result of section \ref{Section3} is that $\B$ is saturated by the center-unstable foliation:  as the unstable holonomy of center leaves is trivial (see Lemma \ref{lemma1}), the stable holonomy groups of points in the same unstable leave are conjugated (see Lemma \ref{lemma2}). We deduce that $\B$ is an attactor. In particular, this implies that $f$ cannot be transitive.  

Section \ref{four} is dedicated to rule out the non transitive case. 

The attractor $\B$ induces an associated repeller $\R$. We first see that $\R$ is saturated by the center foliation, implying that the center leaf volume function is bounded in $\R$ (see section \ref{Rconstruction}). 

We then give a sort of topological description of center-stable leaves in $\R$, namely, all of them are bundles over a center leaf with stable manifolds as fibers, thus having trivial transverse holonomy (see section \ref{completeandbundle}). 

Finally, this allow us to adapt Hiraide arguments \cite{Hir} (see also Bohnet \cite{Boh}) in order to discard the existence of the codimension one transversally unstable repeller $\partial \R$.

\subsection{Acknowledgements} The authors are greately thankful with their advisor Rafael Potrie. This work could not be possible without his advice, patience and generosity. 

The authors would also like to thank Christian Bonatti, Sylvain Crovisier, Sergio Fenley, Andrey Gogolev, Andy Hammerlindl and Martín Sambarino for their generous time and helpful conversations.

\section{Preliminaries}\label{preliminaries}

\subsection{Preliminaries from partially hyperbolic dynamics} 
Let $f:M\to M$ be a $C^1$ diffeomorphism on $M$ a closed Riemannian manifold. We say that $f$ is \emph{partially hyperbolic} if there exists a continuous and $Df$-invariant decomposition $$TM=E^s\oplus E^c\oplus E^u,$$ and some $\ell>0$ such that for every $x\in M$ and unit vectors $v^\sigma\in E^\sigma$ for $\sigma\in \{s,c,u\}$, one has that:
\begin{center} $\|D_xf^{\ell}(v^s)\|<1$, \hspace{0.2cm} $\|D_xf^{-\ell}(v^u)\|<1$\hspace{0.2cm}  and \end{center}

\begin{center} $\|D_xf^{\ell}(v^s)\|<\|D_xf^{\ell}(v^c)\|<\|D_xf^{\ell}(v^u)\|.$ \end{center}

We call \emph{stable}, \emph{center-stable}, \emph{center}, \emph{center-unstable} and \emph{unstable bundle} to $E^s,E^s\oplus E^c,E^c,E^c\oplus E^u$ and $E^u$, respectively.

The stable and unstable bundles are known to be uniquely integrable (see e.g. \cite{HPS}) to foliations $\W^s$ and $\W^u$, respectively. However, the center, center-stable or center-unstable bundles may not integrate. 

We say that a partially hyperbolic diffeomorphism is \emph{dynamically coherent} if the center-stable and center-unstable bundles integrate to $f$-invariant foliations $\W^{cs}$ and $\W^{cu}$, respectively. In particular, this implies that the center bundle is also integrable: the center leaf through a point $x\in M$ being the connected component of $W^{cs}(x)\cap W^{cu}(x)$ that contains $x$. The resulting foliation, $\W^c$, is then also $f$-invariant and tangent to the center bundle. For more information and context on this topics see e.g. \cite{RHRHU}.

Given a point $x\in M$ we denote as $W^\sigma(x)$ the leaf of $\W^\sigma$ through $x$ for $\sigma\in \{s,cs,cu,u\}$. We denote $\C_x$ the leaf of $\W^c$ through $x$. 

Given a center leaf $\C$, we denote $W^\sigma(\C)$ the leaf of $\W^\sigma$ through $\C$ for $\sigma\in\{cs,cu\}$.

For $x\in M$ and $r>0$, we denote $\Ball^\sigma_r(x)\subset W^\sigma(x)$ the intrinsic ball of center $x$ and radius $r>0$ in $W^\sigma(x)$ for $\sigma \in \{s,cs,c,cu,u\}$. If $\dim(E^{\sigma})=1$ we will simply denote it $(x-r,x+r)^{\sigma}$. If $\dim(E^{\sigma})=1$ and $\W^{\sigma}$ is oriented, we denote $(x,x+\delta)^u=\{ y\in \Ball^u_{\delta}(x):x<y\}$ and $W^u_+(x)=\{ y\in W^u(x):x<y\}$, and, if $y\in\W^u(x)$, we denote by $[x,y]^u$ and $(x,y)^u$ the oriented closed and open segments in $\W ^u(x)$, respectively, from $x$ to $y$.

Leaves of $\W^s$ and $\W^u$ can be obtained as increasing union of balls, thus, are homeomorphic to $\mathbb{R}^{\dim E^s}$ and $\mathbb{R}^{\dim E^u}$, respectively.

\subsection{Preliminaries from foliation theory} We will consider continuous foliations with $C^1$-leaves tangent to a continuous distribution. A general reference for this section is \cite{CC}.

\subsubsection{Holonomy of a leaf} 

Let us briefly recall the definition of the holonomy group of a leaf.

Consider $W$  a leaf in a foliation $\W$ of codimension $q$. Fix $x_0$ a point in $W$ and let $D$ be a disk of dimension $q$ transversal to $\W$ through $x_0$.

For every loop $\gamma:[0,1]\to W$ based on $x_0$ one can consider $h_\gamma :D' \to D$ the \emph{holonomy return map} to $D$ of the leaves of $\W$ through points from a smaller transversal disk $D'\subset D$. Consider adequate small transversal disks $D_{\gamma(t)}$ through each point $\gamma(t)$ and, given $y$ in $D'=D_{\gamma(0)}$, define $h_\gamma(y)$ as the end point of the continuous curve $\gamma_y$ defined by $\gamma_y(t)\in D_{\gamma(t)}\cap W(y)$ and $\gamma_y(0)=y$.

For a local homeomorphism $h:D'\subset D\to D$ fixing $x_0$ one defines the \emph{germ of $h$} as the class of all local homeomorphisms that coincides with $h$ in a neighborhood of $x_0$. With the operation given by the composition, this classes form $G(x_0,D)$ the \emph{group of germs of local homeomoprhisms at $x_0$}.

One can see that for basepoint fixed homotopic curves in $W$ based at $x_0$ there corresponds the same holonomy germ. Since concatenation of curves corresponds to composition of holonomy maps (whenever well defined) one has a well defined homomorphism
$$\phi:\pi_1(W,x_0)\to G(x_0,D)$$
where $\pi_1(W,x_0)$ is the fundamental group of $W$ based in $x_0$.

Define the \emph{holonomy group of $W$ at $x_0$} as $\phi(\pi_1(W,x_0))$. The isomorphism class of $\phi(\pi_1(W,x_0))$ does not depend of $x_0$ or $D$, so we call it \emph{holonomy group of $W$} and denote it by $\Hol(W)$.

\subsubsection{Reeb stability}

This next theorem is a classical in foliation theory. See e.g. \cite[Theorem 2.4.3]{CC} and \cite[Theorem 3.]{CLN}.

\begin{teo}[Generalized Reeb stability theorem]\label{Reeb}
Let $W$ be a compact leaf in a foliation $\W$ such that $W$ has a finite holonomy group $\Hol(W)$. Then there exists $\U(W)$ a neighbourhood of $W$, saturated by leaves of $\W$, in which all leaves of $\W$ are compact with finite holonomy group. Moreover, $\U(W)$ has an associated projection $\pi:\U(W)\to W$ such that for every $W' \subset \U(W)$ the map $\pi|_{W'}:W'\to W$ is a finite covering with $k$ sheets, $k\leq |\Hol(W)|$, and for each $y\in W$ the set $\pi ^{-1}(\{y\})$ is a disk transversal to $\W$. The neighbourhood $\U(W)$ can be taken to be arbitrarily small. 
\end{teo}

\subsubsection{Compact foliatons}\label{compactfoliations} 

We say that a foliation $\W$ is \emph{compact} if every leaf $W$ of $\W$ is compact.

Given a compact foliation $\W$ and a Riemannian metric in $M$ we can consider the \emph{volume function} $$\mathrm{vol}:M\to [0,+\infty)$$ that assigns to each point $x\in M$ the volume of the leaf $W(x)$ with respect to the metric in $W(x)$ induced by the metric of $M$.

It may be the case that a compact foliaton does not have uniformly bounded volume of leaves, meaning that $\mathrm{vol}$ is an unbounded function (see for example \cite{S} or \cite{EV}). 

Given a compact foliation $\W$, define the \emph{bad set of $\W$} as $$\B:=\{x\in M: \mathrm{vol} \text{ is not locally bounded at } x\}.$$

The set $\B$ does not depend on the choice of the metric in $M$.

\begin{obs}
The fact that $\W$ has uniformly bounded volume of leaves is equivalent to $\B$ being the empty set.
\end{obs}

\begin{obs}
Observe that if $\W$ is invariant by a $C^1$ diffeomorphism $f$ then $\B$ is invariant by $f$ (namely, $f(\B)=\B$). This is going to be used in our context where $\W$ will be the center foliation of a partially hyperbolic diffeomorphism.
\end{obs}

Let us see some of the properties of $\B$.
For a proof of the following result see, for example, Epstein \cite{Eps} or Lessa \cite{Les}:

\begin{prop} Let $\W$ be a compact foliation. Then the volume function is lower semicontinuous. That is, $\liminf_{x_n\rightarrow x} \mathrm{vol}(W(x_n)) \geq \mathrm{vol}(W(x))$.
\end{prop}

Using the past proposition we can prove (see e.g. \cite{EMS}):

\begin{prop}\label{PropertiesB}
The bad set $\B$ is closed with empty interior and is saturated by leaves of $\W$.
\end{prop}
\begin{proof}

The set $\B$ is clearly closed. 

Semicontinuous functions are continuous in a residual set and $\B\subset\{x\in M: \mathrm{vol} \text{ is not continuous in }x\}$. Thus, from $\mathrm{vol}$ being lower semicontinuous we deduce that $\B$ has empty interior.

The fact that the volume function is constant along leaves implies that $\B$ is saturated by leaves of $\W$: if $x\in \B$ then $x$ has arbitrarily large leaves arbitrarily close to it, so if $y\in W(x)$ then this arbitrarily long leaves pass close to $y$ as well (due to the continuity of the foliation $\W$), meaning that $y$ also belongs to $\B$.
\end{proof}

Leaves in $\B$ can be completely caracterized in terms of holonomy (see e.g. Epstein \cite[Theorem 4.2.]{Eps}):

\begin{prop}\label{Holinfty}
Let $\W$ be a compact foliation and $W$ a leaf of $\W$. Then $W\subset \B$ if and only if $|\Hol(W)|=\infty$.
\end{prop}

\subsection{Stable and unstable holonomies and product neighborhoods}

In this section we see particular definitions and results we will use later.

Throughout this section $f:M\to M$ will be a dynamically coherent partially hyperbolic diffeomorphism with compact center foliation.

\subsubsection{Stable and unstable holonomies}\label{holonomies}

A consequence of dynamical coherence is that each leaf $W$ of $\W^{cs}$ is foliated by center leaves. We denote the restricted foliation by $\W^c|_W$.

Given a center leaf $\C$ denote by $\Hol(\C)$ the holonomy group of $\C$ as a leaf of $\W^c$. 

Denote by $\Hol^s(\C)$ the holonomy group of $\C$ as a leaf of $\W^c|_{W^{cs}(\C)}$.

If $W$ is a leaf of $\W^{cs}$ denote by $\Hol_{cs}(W)$ the holonomy group of $W$ as a leaf of $\W^{cs}$. More generally, if $V\subset W$, we denote $\Hol_{cs}(V)$ the subgroup of $\Hol_{cs}(W)$ that corresponds to holonomy return maps along closed curves inside $V$. 

We analogously define $\Hol^u$ and $\Hol_{cu}$.

\begin{obs}\label{Hol_{cs}=Hol^u} We have that
$\Hol_{cs}(\C)\simeq \Hol^u(\C)$ and $\Hol_{cu}(\C)\simeq \Hol^s(\C)$ for every center leaf $\C$.
\end{obs}
\begin{proof}
We prove the first equality, the second one is analogous.
Let $\gamma$ be a loop in a center leaf $\C$ based in some point $x_0\in \C$. Consider $D^u(x_0)$ a small unstable disk trough $x_0$. One can consider on one hand the holonomy return map $h_\gamma$ associated to $\gamma$ of the leaf $W^{cs}(\C)$ in the foliation $\W^{cs}$ and, on the other hand, the holonomy return map $h_\gamma'$ of the leaf $\C$ in the foliation $\W^c|_{W^u(\C)}$. Dynamical coherence gives us that $h_\gamma$ coincides with $h_\gamma'$. The thesis follows. 
\end{proof}

\begin{obs}
For a set $V$ inside a center-stable leaf $W$ that deformation retracts inside $W$ to some center leaf $\C$, we have that $\Hol_{cs}(V)\simeq\Hol_{cs}(\C)\simeq\Hol^u(\C)$. For such sets we write $\Hol^u(V)$ instead of $\Hol_{cs}(V)$. Analogously for $\Hol_{cu}(V)$ and $\Hol^s(V)$.
\end{obs}

\begin{obs}
Each one of the previous holonomies is invariant by $f$.
\end{obs}

Dynamical coherence gives a relationship between the holonomy $\Hol(\C)$ and the holonomies $\Hol^s(\C)$ and $\Hol^u(\C)$. In particular, we are going to need the following (see e.g. Carrasco \cite[Proposition 2.5.]{Car}):
\begin{prop}\label{Holprod} For every center leaf $\C$ we have that $$\max\{|\Hol^s(\C)|,|\Hol^u(\C)|\}\leq |\Hol(\C)| \leq |\Hol^s(\C)||\Hol^u(\C)|.$$

In particular, $\Hol(\C)$ is finite if and only if $\Hol^s(\C)$ and $\Hol^u(\C)$ are finite.
\end{prop}

\begin{proof} Let $\C$ be a center leaf and $x_0$ a point in $\C$.

Consider $D$ a small transversal disk to $\W^c$ and tangent to $E^s\oplus E^u$ in $x_0$. Let us denote $D^s$ and $D^u$ the connected components of $D \cap W^{cs}(x_0)$ and $D \cap W^{cu}(x_0)$ that contains $x_0$, respectively. More general, for every $x$ in $D$ denote $D^s(x)$ and $D^u(x)$ the connected components of $D \cap W^{cs}(x)$ and $D \cap W^{cu}(x)$ that contains $x$, respectively.

We have then, in a neighborhood $D'\subset D$ of $x_0$, local stable/unstable coordinates. That is, for every $x$ in $D'$ there exists unique $x_s\in D^s$ and $x_u\in D^u$ such that $x=D^u(x_s)\cap D^s(x_u)$. Denote the coordinates of $x$ by $(x_s,x_u)$.

Given a loop $\gamma$ based at $x_0$ we have well defined holonomy return maps, $h_\gamma^s:{D^s}'\to D^s$ and $h_\gamma^u:{D^u}'\to D^u$, for center leaves inside $W^{cs}(x_0)$ and $W^{cu}(x_0)$, respectively. Dynamical coherence says then that the holonomy return map associated to $\gamma$ in $D$ can be written in local cordinates as: $$h_\gamma=(h_\gamma^s,h_\gamma^u),$$ where $h_\gamma$ is given by $h_\gamma(x)=(h_\gamma^{s}(x_s),h_\gamma^{u}(x_u))$, with $x=(x_s,x_u)$ the stable/unstable coordinates defined before.

Elements of $\Hol^s(\C)$ and $\Hol^u(\C)$ are germs of local homeomorphism fixing $x_0$ in $D^s$ and $D^u$, respectively. The previous discussion shows that the map
$$\psi:\Hol^s(\C)\times \Hol^u(\C) \to \Hol(\C)$$

given by $\psi(h_\gamma^s,h_\gamma^u)=h_\gamma$ is surjective. This implies that $|\Hol(\C)| \leq |\Hol^s(\C)||\Hol^u(\C)|$.

Moreover, since $h_\gamma(x_s,x_0)=(h^s(x_s),x_0)$ and $h_\gamma(x_0,x_u)=(x_0,h^u(x_u))$, we obtain that $\max\{|\Hol^s(\C)|,|\Hol^u(\C)|\}\leq |\Hol(\C)| $.
\end{proof}

\subsubsection{Product neighborhoods} 

Assume throughout the rest of this section that $\dim(E^u)=1$ and $\Hol^u(\C)=\Id$ for every center leaf. This hypothesis will be fulfilled later during the proof of the main theorem.

The following will be of good use many times (for similar results see e.g. \cite{CC}) or \cite{HH}:

\begin{prop}[Product neighborhoods]\label{productnbhd}
Let $\C$ be a center leaf and consider $E=\bigcup_{x\in \C}\Ball^s(x)$, where $\Ball^s(x)\subset W^s(x)$ is a disk such that $\Ball^s(x)\cap \C = \{x\}$ for every $x\in \C$ and $\Ball^s(x)$ varies continuously with $x$. Then there exists an homeomorphism over its image $\varphi:E\times [-1,1]\to M$ such that:
\begin{enumerate}
\item $\varphi(E\times \{0\})=E$.
\item $\varphi(E \times \{y\})$ lies inside a $\W^{cs}$ leaf for every $y\in [-1,1]$.
\item $\varphi(\{x\}\times [-1,1])$ lies inside a $\W^u$ leaf for every $x\in E$.
\end{enumerate}
\end{prop}

In this case, $\varphi(E\times [-1,1])$ will be called as a \emph{$(\W^{cs},\W^u)$-product neighbourhood} of $E$ and, by an abuse of notation, we will simply refer to it as $E\times [-1,1]$.

\begin{proof}
We can consider $\epsilon>0$ such that for every distinct $x,x'\in E$ we get $(x-\epsilon,x+\epsilon)^u\cap (x'-\epsilon,x'+\epsilon)^u=\emptyset$.

Let us fix $x_0\in \C$. By the continuity of $\W^{cs}$ we can consider $\delta>0$ such that for every curve $\gamma:[0,1]\to E$ with $\gamma(0)=x_0$ and $\length(\gamma)\leq 2\diam(E)$ we have a well defined holonomy map $$h_\gamma:(x_0-\delta,x_0+\delta)^u\to (\gamma(1)-\epsilon,\gamma(1)+\epsilon)^u.$$ 

Observe that if such a $\gamma$ is closed, since $E=\bigcup_{x\in \C}\Ball^s(x)$ with each $\Ball^s(x)$ a stable disk, then $\gamma$ would be homotopic to a loop in $\C$. Since $\Hol^u(\C)=\Id$ we can take $\delta>0$ small enough as to assure that $h_\gamma$ coincides with the identity from $(x_0-\delta,x_0+\delta)^u$ to itself for every $\gamma$ such that $\length(\gamma)\leq 2\diam(E)$.

For every $x\in E$ consider $\gamma_x:[0,1]\to E$ a curve from $x_0$ to $x$ such that $\length(\gamma_x)\leq \diam(E)$. Now, for every $y\in (x_0-\delta,x_0+\delta)^u$ define $$\varphi(x,y)=h_{\gamma_x}(y).$$

This definition is independent of the choice of the curve $\gamma_x$. Indeed, if $\gamma_x':[0,1]\to E$ is another such a curve, the concatenation $\gamma_x*\gamma_x'$ has length at most $2\diam(E)$ and then $id=h_{\gamma_x*\gamma_x'}=h_{\gamma_x}\circ h_{\gamma_x'}^{-1}$ implies $h_{\gamma_x}(y)=h_{\gamma_x'}(y)$ for every $y\in (x_0-\delta,x_0+\delta)^u$.

The properties of the map $\varphi$ follow directly from its definition

\end{proof}

We will many times need a particular instance of the previous proposition: 

\begin{obs}\label{Reebproduct}
For every center leaf $\C$ such that $|\Hol(\C)|<\infty$ we can consider $\U^s(\C)$ a neighbourhood of $\C$ in $W^{cs}(\C)$ given by the Generalized Reeb stability theorem (see Theorem \ref{Reeb}) such that:
\begin{itemize}
\item The associated projection $\pi:\U^s(\C)\to \C$ satisfies that $\pi^{-1}(x)$ is a disk in $W^s(x)$ for every $x\in \C$.

\item There exists a $(\W^{cs},\W^u)$-product neighbourhood $\U^s(\C)\times [-1,1]$ of $\U^s(\C)$.
\end{itemize}
\end{obs}
\begin{proof}
By the transversality of the foliations $\W^s$ and $\W^c$ inside $W^{cs}(\C)$ and the fact that $\U^s(\C)$ can be taken arbitrarily small we obtain that the projection $\pi$ can be taken along leaves of $\W^s$.

The neighbourhood $\U^s(\C)$ fulfills the hypothesis of Proposition  \ref{productnbhd} and that implies the existence of the product neighbourhood.
\end{proof}

\section{The bad set is saturated by the center-unstable foliation}\label{Section3}

From now on, let $f:M\to M$ be a dynamically coherent partially hyperbolic diffeomorphism such that $\dim(E^u)=1$.

As before, we assume that $\W^c$ is a compact foliation. We are going to see that $\W^c$ is in fact uniformly compact (meaning that the leaf volume function is bounded in $M$).

\begin{obs} We can assume from now on that all bundles $E^s,E^{cs},E^c,E^{cu}$ and $E^u$ are orientable. 
\end{obs}
\begin{proof}
By taking a finite cover of $M$ we can lift all bundles $E^s,E^{cs},E^c,E^{cu}$ and $E^u$ to orientable bundles. Then $f$ lifts to a dynamically coherent partially hyperbolic diffeomorphism $\tilde{f}$ whose center foliation $\tilde{\W}^c$ is the lift of $\W^c$. The lifted center foliation $\tilde{\W}^c$ remains compact, and each one of its leaves is a finite cover of some leaf of $\W^c$. Then, if $\tilde{\W}^c$ is uniformly compact, so is $\W^c$.
\end{proof}

Let us see first a simple but yet crucial consequence of the codimension one hypothesis:

\begin{lema}\label{lemma1} 
The group $\Hol^u(\C)$ is trivial for every center leaf $\C$.
\end{lema}
\begin{proof} 

Let $\C$ be a center leaf. Recall that $\Hol^u(\C)$ consists of the holonomy maps associated to $\C$ inside the center-unstable leaf $W^{cu}(\C)$.

Let $x$ be a point in $\C$ and $\gamma:[0,1]\to \C$ a closed curve based in $x$. Consider $\lambda' \subset\lambda\subset \W^u(x)$ small enough one-dimensional transversals through $x$ such that the holonomy return map associated to $\gamma$ is a well defined map $h_\gamma:\lambda' \to \lambda$. As we are assuming that the unstable foliation is orientable, the map $h_\gamma$ preserves the orientation of $\lambda$. 

Assume that $h_\gamma$ is not the identity. Then there exists $y\in\lambda'$ such that $\{h_\gamma^n(y)\}_{n\geq 0}$ or $\{h_\gamma^{-n}(y)\}_{n\geq 0}$ constitutes an infinite set of points lying in $\lambda$. Assume without loss of generality it is the former.

Fix $U$ a small foliated neighborhood of $\W^c$ containing $\lambda$. Through every point of $\{h_\gamma^n(y)\}_{n\geq 0}$ corresponds a different plaque of $U$. Since each one of this plaques belong to $\C_y$ this contradicts the fact that $\C_y$ is compact. 
\end{proof}

From the previous lemma, we deduce that each center leaf $\C$ has in $W^{cu}(\C)$ a product neighbourhood of the form $\C\times (-\delta,\delta)^u$ , where each $\C \times \{y\}$ corresponds to a center leaf and each $\{x\} \times (-\delta, \delta)^u$ corresponds to an unstable arc (see Figure \ref{fig1}). This will allow us to use Remark \ref{Reebproduct}.

This kind of ``stacking'' of center leaves along the unstable direction implies that stable holonomy is constant along unstable leaves (see Figure \ref{fig2}):

\begin{lema}\label{lemma2} For every $x\in M$ and $y\in W^u(x)$ the groups $\Hol^s(\C_x)$ and $\Hol^s(\C_y)$ are isomorphic.
\end{lema}

\begin{proof}
Obseve that it is enough to give a local argument: suppose that for every $x\in M$ there is an unstable arc $(x-\delta, x+\delta)^u$ such that $\Hol(\C_y)\simeq \Hol(\C_{x})$ for every $y\in (x-\delta, x+\delta)^u$. This implies that the set $\{y\in W^u(x)$ :  $\Hol(\C_y)$ is isomorphic to $\Hol(\C_x)\}$ is an open subset of $W^u(x)$ as well as its complement, and thus the thesis follows. 

Let $x$ be a point in $M$ and denote by $\C$ the center leaf through $x$. Since $\C$ is compact there exists $\delta>0$ such that $\Ball^s_\delta(z)\cap \Ball^s_\delta(z')=\emptyset$ for every distinct $z,z'\in \C$. Thus $\bigcup_{z\in \C}\Ball^s_\delta(z)$ (denote it $\Ball^s_\delta(\C)$) is in the hypothesis of the Proposition \ref{productnbhd} and we can consider a $(\W^{cs},\W^u)$-product neighbourhood $\Ball^s_\delta(\C)\times (-1,1)$. 

Given $t\in (-1,1)$ we can define a projection $p_{t}:\Ball^s_\delta(\C)\times \{0\}\to \Ball^s_\delta(\C)\times \{t\}$ along unstable leaves inside $\Ball^s_\delta(\C)\times (-1,1)$. Namely $$p_{t}(z,0)=(z,t).$$

The projection $p_{t}$ then indentifies $\W^c|_{\Ball^s_\delta(\C)\times \{0\}}$ homeomorphically with $\W^c|_{\Ball^s_\delta(\C)\times \{t\}}$ since the leaves of $\W^c$ are the connected components of the intersection of leaves of $\W^{cs}$ with leaves of $\W^{cu}$. This implies that 
\begin{center} $\Hol^s(\C_{y})\simeq\Hol^s(\C_{x})$ \end{center} 
for every $y\in \{ x\}\times (-1,1)$.

%$\W^c=\W^{cs}\cap \W^{cu}$ is that for every $z\in M$ the leaf $\C_z$ is the connected component of $W^{cs}(z)\cap W^{cu}(z)$ that contains $z$. 

%In fact, let $z$ be any point in $W^s_\delta(\C)\times \{y\}$ and denote $\tilde{\C_z}$ the connected component of $\C_z\cap W^s_\delta(\C)\times \{y\}$ that contains $z$. Then $p_{yy''}(\tilde{\C_z})$ is $W^u_{loc}(\tilde{\C_z})\cap (W^s_\delta(\C)\times \{y'\})$, wich implies that $p_{yy''}(\tilde{\C_z})\subset \C_{p_{yy'}(z)}$..

\end{proof}

\begin{figure}[!tbp]
\centering
\begin{minipage}[b]{.5\textwidth}
  \centering
%% Accompanies image file 'Figure1.pdf' (pdf, eps, ps)
\begingroup%
  \makeatletter%
  \providecommand\color[2][]{%
    \errmessage{(Inkscape) Color is used for the text in Inkscape, but the package 'color.sty' is not loaded}%
    \renewcommand\color[2][]{}%
  }%
  \providecommand\transparent[1]{%
    \errmessage{(Inkscape) Transparency is used (non-zero) for the text in Inkscape, but the package 'transparent.sty' is not loaded}%
    \renewcommand\transparent[1]{}%
  }%
  \providecommand\rotatebox[2]{#2}%
  \newcommand*\fsize{\dimexpr\f@size pt\relax}%
  \newcommand*\lineheight[1]{\fontsize{\fsize}{#1\fsize}\selectfont}%
  \ifx\svgwidth\undefined%
    \setlength{\unitlength}{97.12747209bp}%
    \ifx\svgscale\undefined%
      \relax%
    \else%
      \setlength{\unitlength}{\unitlength * \real{\svgscale}}%
    \fi%
  \else%
    \setlength{\unitlength}{\svgwidth}%
  \fi%
  \global\let\svgwidth\undefined%
  \global\let\svgscale\undefined%
  \makeatother%
  \begin{picture}(1,1.22289996)%
    \lineheight{1}%
    \setlength\tabcolsep{0pt}%
    \put(0,0){\includegraphics[width=\unitlength,page=1]{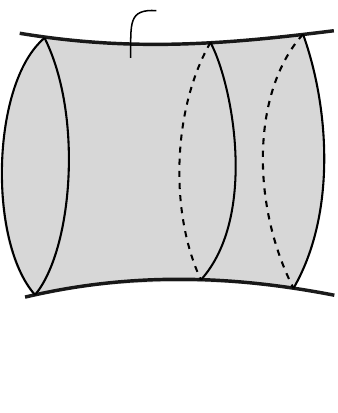}}%
    \put(0.47241561,1.15732117){\color[rgb]{0,0,0}\makebox(0,0)[lt]{\lineheight{1.25}\smash{\begin{tabular}[t]{l}$\C$\end{tabular}}}}%
    \put(0,0){\includegraphics[width=\unitlength,page=2]{Figure1.pdf}}%
    \put(0.82761059,1.16284565){\color[rgb]{0,0,0}\makebox(0,0)[lt]{\lineheight{1.25}\smash{\begin{tabular}[t]{l}$\ii$\end{tabular}}}}%
    \put(0,0){\includegraphics[width=\unitlength,page=3]{Figure1.pdf}}%
  \end{picture}%
\endgroup%

  \caption{}\label{fig1}
  \label{fig:test1}
\end{minipage}%
\hfill
\begin{minipage}[b]{.5\textwidth}
  \centering
%% Accompanies image file 'Figure2.pdf' (pdf, eps, ps)
\begingroup%
  \makeatletter%
  \providecommand\color[2][]{%
    \errmessage{(Inkscape) Color is used for the text in Inkscape, but the package 'color.sty' is not loaded}%
    \renewcommand\color[2][]{}%
  }%
  \providecommand\transparent[1]{%
    \errmessage{(Inkscape) Transparency is used (non-zero) for the text in Inkscape, but the package 'transparent.sty' is not loaded}%
    \renewcommand\transparent[1]{}%
  }%
  \providecommand\rotatebox[2]{#2}%
  \newcommand*\fsize{\dimexpr\f@size pt\relax}%
  \newcommand*\lineheight[1]{\fontsize{\fsize}{#1\fsize}\selectfont}%
  \ifx\svgwidth\undefined%
    \setlength{\unitlength}{130.85201636bp}%
    \ifx\svgscale\undefined%
      \relax%
    \else%
      \setlength{\unitlength}{\unitlength * \real{\svgscale}}%
    \fi%
  \else%
    \setlength{\unitlength}{\svgwidth}%
  \fi%
  \global\let\svgwidth\undefined%
  \global\let\svgscale\undefined%
  \makeatother%
  \begin{picture}(1,1.32817258)%
    \lineheight{1}%
    \setlength\tabcolsep{0pt}%
    \put(9.9126172,-0.3239316){\color[rgb]{0,0,0}\makebox(0,0)[lt]{\begin{minipage}{6.12794684\unitlength}\raggedright \end{minipage}}}%
    \put(-2.42847343,3.84516575){\color[rgb]{0,0,0}\makebox(0,0)[lt]{\begin{minipage}{21.03283703\unitlength}\raggedright \end{minipage}}}%
    \put(0.6970723,0.01443673){\color[rgb]{0,0,0}\makebox(0,0)[lt]{\lineheight{1.25}\smash{\begin{tabular}[t]{l}$\C_x$\end{tabular}}}}%
    \put(0.13135317,1.09765927){\color[rgb]{0,0,0}\makebox(0,0)[lt]{\lineheight{1.25}\smash{\begin{tabular}[t]{l}$\C_y$\end{tabular}}}}%
    \put(-0.00508125,0.15255844){\color[rgb]{0,0,0}\makebox(0,0)[lt]{\lineheight{1.25}\smash{\begin{tabular}[t]{l}$W^{cs}(\C_x)$\end{tabular}}}}%
    \put(0.67950996,1.03199722){\color[rgb]{0,0,0}\makebox(0,0)[lt]{\lineheight{1.25}\smash{\begin{tabular}[t]{l}$\ii$\end{tabular}}}}%
    \put(0,0){\includegraphics[width=\unitlength,page=1]{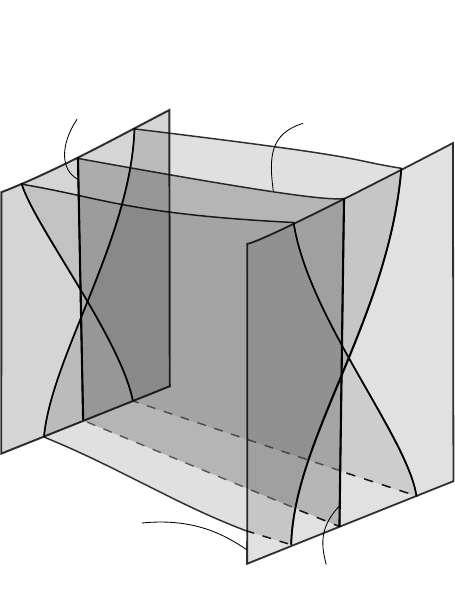}}%
    \put(0.38032448,1.13361453){\color[rgb]{0,0,0}\makebox(0,0)[lt]{\lineheight{1.25}\smash{\begin{tabular}[t]{l}$W^{cs}(\C_y)$\end{tabular}}}}%
    \put(0,0){\includegraphics[width=\unitlength,page=2]{Figure2.pdf}}%
  \end{picture}%
\endgroup%

  \caption{}\label{fig2}
  \label{fig:test2}
\end{minipage}
\end{figure}
Recall from the preliminaries that we denote the bad set $\B$ of $\W^c$ as the points of $M$ in which the leaf volume function is not locally bounded.

\begin{cor}\label{saturated} The bad set $\B$ of $\W^c$ is saturated by the center-unstable foliation. 
\end{cor}
\begin{proof}
From Proposition \ref{Holinfty} we have that a center leaf belongs to $\B$ if and only if $|\Hol(\C)|=\infty$. By Proposition \ref{Holprod} and Lemma \ref{lemma1}, we have that $|\Hol(\C)|< \infty$ if and only if $|\Hol^s(\C)|< \infty$.

The previous two lemmata then implies that $\B$ is saturated by the unstable foliation. As it is also saturated by the center foliation (see Remark \ref{PropertiesB}) the thesis follows.
\end{proof}

We obtain the following.

\begin{cor} The bad set $\B$ of $\W^c$ is a proper attractor. In particular, there are no transitive, codimension one, dynamically coherent partially hyperbolic diffeomorphisms with compact center foliation and unbounded volume of leaves.
\end{cor}
\begin{proof}
The set $\B$ is compact, $f$-invariant, has empty interior and is saturated by the center-unstable foliation (see Corollary \ref{saturated}). In particular, $\B$ is transversally stable and not all $M$, and thus a proper attractor. This implies that $f$ cannot be transitive.
\end{proof}

The aim of the rest of the work is to see that same thesis follows in the non transitive scenario.

We finish this section noting that, by what we have done by this point and the work of Gogolev in \cite{Gog}, one is already able to discard the Sullivan foliation \cite{S} as being the center foliation of a dynamically coherent partially hyperbolic system:

\begin{obs}\label{NoSullivan}
The Sullivan foliation \cite{S} cannot be the center foliation of a dynamically coherent partially hyperbolic diffeomorphism.
\end{obs}
\begin{proof}
The example given by Sullivan is a foliation by circles in a five dimensional compact space $M$ with unbounded length of leaves. 

Assume that Sullivan's foliation is the center foliation of a dynamically coherent partially hyperbolic diffeomorphism.

As the center foliation is one dimensional (in a five dimensional manifold), Gogolev's work implies that it has uniformly bounded volume of leaves if $\dim(E^s)= 2$ and $\dim(E^s)= 2$ (see Main Theorem in \cite{Gog}).

It remains to rule out the codimension one case. Assume without loss of generality that $\dim(E^u)=1$.

Let us denote the Sullivan foliation by $\F$. In this particular case, the bad set $\B$ of $\F$ has the structure of the unitary tangent bundle $T^1S^2$ of a $2$-sphere. Moreover, the leaves of the foliation in $\B$ are exactly the fibers of this unitary tangent bundle. 

By the Corollary \ref{saturated} the set $\B$ is saturated by the center-unstable foliation, and so center-unstable leaves foliate $\B$.

Given a center leaf $\C$ we have that $\Hol^u(\C)=\Id$ and then $\C$ has a neighbourhood $\C\times (-\delta,\delta)^u$ in $W^{cu}(\C)$ such that each $\C\times \{y\}$ is a center leaf (see Figure \ref{fig1}).

This implies that the center-unstable foliation in $\B\simeq T^1S^2$ projects to a (topological) foliation without singularities in the base $S^2$. This cannot be.
\end{proof}

\section{Proof of the non transitive case}\label{four}

We saw in the previous section that $f$ cannot be transitive if the bad set $\B$ of $\W^c$ is not empty. Lose the transitivity hypothesis and assume that $\B$ is not empty. We are going to see that this yields a contradiction.

\subsection{Construction of the repeller $\R$}\label{Rconstruction}

Let us consider $\R$ the repelling set induced by the attactor set $\B$ (see Figure \ref{figR}): $$\R=M \setminus \bigcup_{x \in \B}W^s(x).$$ 

We will next closely study $\R$. In particular, we are going to see that is saturated by center-stable leaves.

\vspace{0.2cm}
\begin{figure}[htb]
  \centering
  \def\svgwidth{200pt} 
%% Accompanies image file 'Figure3.pdf' (pdf, eps, ps)
\begingroup%
  \makeatletter%
  \providecommand\color[2][]{%
    \errmessage{(Inkscape) Color is used for the text in Inkscape, but the package 'color.sty' is not loaded}%
    \renewcommand\color[2][]{}%
  }%
  \providecommand\transparent[1]{%
    \errmessage{(Inkscape) Transparency is used (non-zero) for the text in Inkscape, but the package 'transparent.sty' is not loaded}%
    \renewcommand\transparent[1]{}%
  }%
  \providecommand\rotatebox[2]{#2}%
  \newcommand*\fsize{\dimexpr\f@size pt\relax}%
  \newcommand*\lineheight[1]{\fontsize{\fsize}{#1\fsize}\selectfont}%
  \ifx\svgwidth\undefined%
    \setlength{\unitlength}{391.3475724bp}%
    \ifx\svgscale\undefined%
      \relax%
    \else%
      \setlength{\unitlength}{\unitlength * \real{\svgscale}}%
    \fi%
  \else%
    \setlength{\unitlength}{\svgwidth}%
  \fi%
  \global\let\svgwidth\undefined%
  \global\let\svgscale\undefined%
  \makeatother%
  \begin{picture}(1,0.47513046)%
    \lineheight{1}%
    \setlength\tabcolsep{0pt}%
    \put(0,0){\includegraphics[width=\unitlength,page=1]{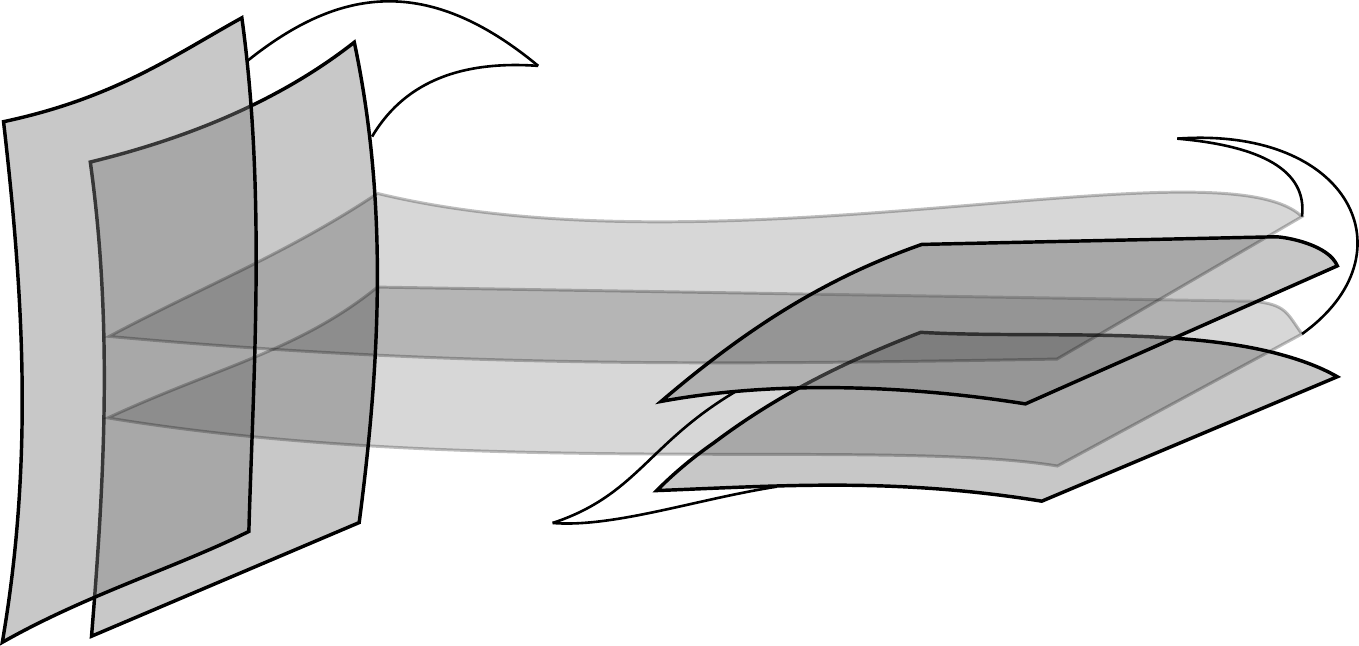}}%
    \put(0.34479636,0.07862604){\color[rgb]{0,0,0}\makebox(0,0)[lt]{\lineheight{1.25}\smash{\begin{tabular}[t]{l}$\R$\end{tabular}}}}%
    \put(0.40566243,0.4059389){\color[rgb]{0,0,0}\makebox(0,0)[lt]{\lineheight{1.25}\smash{\begin{tabular}[t]{l}$\B$\end{tabular}}}}%
    \put(0.69636654,0.36060133){\color[rgb]{0,0,0}\makebox(0,0)[lt]{\lineheight{1.25}\smash{\begin{tabular}[t]{l}$W^{s}(\B)$\end{tabular}}}}%
  \end{picture}%
\endgroup%

  \caption{}\label{figR}
\end{figure}
\vspace{0.2cm}

For every $r>0$, let us denote $\Ball^s_r(\B)=\bigcup_{x \in \B} \Ball^s_r(x)$.

\begin{lema} The set $\R$ is non-empty, $f-$invariant, compact and saturated by the stable foliation. 
\end{lema}
\begin{proof}
Since $\B$ is saturated by leaves of $\W^{cu}$ then $\bigcup_{x \in \B}W^s(x)$ is open. This implies that $\R$ is compact.

Let us see that $\R$ is not empty. The set $\bigcup_{x \in \B}W^s(x)$ is the increasing union of the open sets $\{\bigcup_{x \in \B} \Ball^s_n(x)\}_n$. If it were the case that $M=\bigcup_{x \in \B}W^s(x)$, then $M$ would coincide with $\Ball^s_{n_0}(\B)$ for some $n_0$. Then $f^{-n}(M)$ would be contained in a small neighbourhood of the proper compact subset $\B$ for some big enough $n$ which of course cannot be.

Since $\bigcup_{x \in \B}W^s(x)$ is $f$-invariant and saturated by $\W^s$ then the same is the case for $\R$.

\end{proof}

We are on the way to prove that $\R$ is also saturated by the center foliation. The proof will rely on the following three lemmata.

The next lemma is the main observation that will allow us to continue to work in a neighbourhood of $\R$ as if the center foliation were uniformly compact. 
  
\begin{lema}\label{Blargas} For every $\epsilon >0$ there exists $K>0$ such that if $\C$ is a center leaf with $\mathrm{vol}(\C)>K$ then $\C\subset \Ball^s_\epsilon(\B)$. 
\end{lema}

\begin{proof}
Fix $\epsilon >0$ and suppose that there is no such $K$. Then there exists a sequence $(x_n)_{n\in \mathbb{N}}$ in $M\setminus \Ball^s_\epsilon(\B)$ such that $\mathrm{vol}(\C_{x_n})\xrightarrow{n}\infty$. By taking a convergent subsequence $x_{n_k}\xrightarrow{k} x\in M\setminus \Ball^s_\epsilon(\B)$ we obtain that $\mathrm{vol}$ is not locally bounded in $x$. So $x$ should be a point of $\B$ but that is impossible since $x\in M\setminus \Ball^s_\epsilon(\B)$
\end{proof}

The following lemma relates diameter and volume of leaves:

\begin{lema}\label{voldiam} For every $K>0$ there exists $D>0$ such that if $\C$ is a center leaf with $\mathrm{vol}(\C)<K$ then $\diam(\C)<D$.
\end{lema}
\begin{proof}
Take $\{U_i\}_{i=1,\dotsc,l}$ a finite covering of $M$ by foliated boxes of the center foliation. For each $i\in \{1,\dotsc,l\}$ denote $d_i$ the supremum of the diameter of center plaques in $U_i$ and $v_i$ the infimum of the volume of center plaques in $U_i$. Denote $d=\max_i \{d_i\}$ and $v=\min_i \{v_i\}$. 

Now let $\C$ be a center leaf with $\mathrm{vol}(\C)<K$. We have that $\C$ has less than $\frac{K}{v}+1$ plaques in each $U_i$ and then $\diam(\C)<ld (\frac{K}{v}+1)=D$.
\end{proof}

From the continuity of $\W^c$ we have:

\begin{lema}\label{continuosW^c}
Given $D>0$ and $\epsilon>0$ there exists $\delta>0$ such that for any $x$ and $y$ with $\mathrm{d}(x,y)<\delta$ we have that $\Ball^c_D(y)\subset \Ball_\epsilon (\C_x)$.
\end{lema}

We can now prove:

\begin{prop}\label{csaturated} The set $\R$ is saturated by the center foliation. 
\end{prop}
\begin{proof}
Suppose that there exists $x\in \R$ and $y\in \C_x$ such that $y\in M\setminus \R$. Since $y\in M\setminus \R$ there exists $w\in \B$ such that $y\in W^s(w)$.

Denote $d=\mathrm{d}(\B,\R)>0$. Note by Lemma  \ref{Blargas} that $\mathrm{vol}$ is bounded in $\R$, say by some constant $K>0$. 

By Lemma \ref{voldiam} there exists $D>0$ such that for every center leaf $\C$ with $\vol(\C)<K$ the diameter of $\C$ is less than $D$. So for every $z\in \R$ we have that $\Ball^c_D(z)=\C_z$.

We can now consider $N$ large enough so that, by Lemma \ref{continuosW^c}, the points $f^N(w)$ and $f^N(y)$ are as close as to assure that $\Ball^c_D(f^N(y))\subset \Ball_{d/2}(\B)$. This yields a contradiction since $f^N(y)\in \C_{f^N(x)}$ and $f^N(x)\in \R$ since $\R$ is $f$-invariant. This shows that for every $x\in \R$ the leaf $\C_x \subset \R$.
\end{proof}

\subsection{Completeness and trivial holonomy for center-stable leaves in $\R$}\label{completeandbundle}

In this section we prove some properties of $\R$ in order to implement the proof of Hiraide in Section \ref{Hirai}.

Let us first see in the following proposition that center-stable leaves in $\R$ are \textit{complete} (this terminology is used in \cite{BoW}, \cite{Car} and \cite{BoBo}).

\begin{prop}\label{completeness}
For every center leaf $\C$ in $\R$ we have that $W^{cs}(\C)=\bigcup_{x\in \C}W^s(x)$.
\end{prop}
\begin{proof}
Let $\C$ be a center leaf in $\R$. If we prove that $\bigcup_{x\in \C}W^s(x)$ is saturated by center leaves then $\bigcup_{x\in \C}W^s(x)$ will be a non empty  open and closed subset of $W^{cs}(\C)$ and then will coincide with $W^{cs}(\C)$.

Let us consider $y_0\in  \bigcup_{x\in \C}W^s(x)$. We want to see that $\C_{y_0}\subset \bigcup_{x\in \C}W^s(x)$. Let $x_0$ be any point in $\C$ and let $\gamma:[0,1]\to W^{cs}(\C)$ be a continuous path from $x_0$ to $y_0$.

For every $t\in [0,1]$ denote by $\C_t$ the center leaf trough $\gamma(t)$.

Recall that by Proposition \ref{Holinfty} the center leaves in $\B$ coincide with those with infinite holonomy group. Thus, since $\R\cap \B=\emptyset$, for every $t\in [0,1]$ we have that $\C_t\subset W^{cs}(\C)$ satisfies $|\Hol^s(\C_t)|< \infty$. 

We can then take for each $\C_t$ a neighborhood $\U^s(\C_t)$ of $\C_t$ in $W^{cs}(\C_t)$ given by the Generalized Reeb stability theorem (see Theorem \ref{Reeb}). The associated projection $\pi_t:\U^s(\C_t)\to \C_t$ can be taken such that $\pi_t^{-1}(x)$ is a disk in $W^s(x)$ for every $x\in \C_t$. (See Remark \ref{Reebproduct}). 

Then $\{\U^s(\C_t)\}_{t\in [0,1]}$ is an open cover of $\gamma([0,1])$. Let us take a finite subcover $\{\U^s(\C_{t_0}),\dotsc,\U^s(\C_{t_k})\}$ such that $\C_{x_0}=\C_{t_0}$, $\C_{y_0}=\C_{t_k}$ and $\U^s(\C_{t_i})\cap \U^s(\C_{t_{i+1}})\neq \emptyset$ for every $0\leq i \leq k-1$. 

Observe that, if $\C'$ is a center leaf in some $\U^
s(\C_{t_i})$, then each stable disk of a point of $\C_{t_i}$ intersects $\C'$. 

Observe also that since each $\U^s(\C_{t_i})$ is saturated by center leaves, so it is each $\U^s(\C_{t_i})\cap \U^s(\C_{t_{i+1}})$.

Then, by taking $\C'_i\subset \U^s(\C_{\gamma(t_i)})\cap \U^s(\C_{\gamma(t_{i+1})})$ we deduce that each stable leaf of $\C_i$ intersects $\C_{i+1}$. 

This implies that $\C_{y_0}\subset \bigcup_{x\in \C}W^s(x)$ as we wanted.
\end{proof}

The following is a mild extension of Lemma \ref{Blargas} that will come handy later.

\begin{lema}
There exists a constant $C>0$ such that for every $x$ in $\R$ we have $\#\{\C_x\cap W^s(x)\}< C$.
\end{lema}
\begin{proof}
Cover $\R$ by a finite number $\{U_i\}_{1\leq i \leq k}$ of foliated boxes for the $\W^c$ foliation. Let $d>0$ be such that each plaque of each $U_i$ has volume larger than $d$.

By Lemma \ref{Blargas} the volume function is bounded in $\R$, say by some constant $K$.  

Let us see that $\#\{\C_x\cap W^s(x)\}< K/d +1$ for every $x\in \R$.

Suppose there exists $x\in \R$ and distinct points $\{x=x_0,\dotsc,x_l\}\subset \{\C_x\cap W^s(x)\}$ with $l \geq K/d +1$. 

Let $\gamma>0$ be a Lebesgue number for the covering $\{U_i\}_{1\leq i \leq k}$. Then we can consider $N$ large enough as to assure that $\diam(\{f^N(x_0),\dotsc,f^N(x_l)\})<\gamma$ in $W^s(f^N(x_0))$ with the intrinsic topology. So $\{f^N(x_0),\dotsc,f^N(x_l)\}$ is contained in some member $U_{i_0}$ of the covering. 

Since the points $\{f^N(x_0),\dotsc,f^N(x_l)\}$ are close in $W^s(f^N(x_0))$ in the intrinsic topology, then trough each one of them there correspond a different plaque of $U_{i_0}$. On the other hand, $\{f^N(x_0),\dotsc,f^N(x_l)\}\subset f^N(\C_x)$. This contradicts the fact that $\vol (f^N(\C_x)) \leq K$.

\end{proof}

We can now give some kind of a description of center-stable leaves in $\R$ (see Bohnet \cite[Corollary 4.10.]{Boh}  for a similar result):

\begin{prop}\label{csleavesR} %For every center-stable leaf $W$ in $\R$ there exists a center leaf $\C\subset W$ such that $W^s(x)\cap \C=\{x\}$ for every $x\in \C$. 

Let $W$ be a center-stable leaf in $\R$. Then there exists a center leaf $\C$ in $W$ such that for every $x\in\C$ we have that $\C \cap W^s(x)=\{x\}$. Therefore, $W$ is a bundle with base $\C$ and fibers $\{W^s(x)\}_{x\in \C}$.
\end{prop}
\begin{proof}

Observe that it is enough to prove what we want for some $f^N(W)$.

Cover $\R$ by finite $(\W^{cs},\W^u)$-product neighborhood $\{\U^s(\C_i)\times [-1,1]\}_{1\leq i \leq k}$ with each $\U^s(\C_i)$ being a Generalized Reeb stability neighborhood of the center leaf $\C_i$ (see Remark \ref{Reebproduct}). 

Let $\gamma>0$ be a Lebesgue number for the covering.

Let $W$ be a center-stable leaf in $\R$. By the previous lemma, we have that $\#\{\C_x\cap W^s(x)\}$ is bounded by a constant $C>0$ for every $x\in W$. So, let us consider $x_0\in W$ such that $l=\#\{\C_{x_0}\cap W^s(x_0)\}$ is maximal in $W$. Let $\{x_0,\dotsc,x_{l-1}\}=\C_{x_0}\cap W^s(x_0)$.

Let $N>0$ be large enough so that $\diam(\{f^N(x_0),\dotsc,f^N(x_{l-1})\})<\gamma$ in $W^s(f^N(x_0))$ with the intrinsic topology. Then there exists $i\in \{1,\dotsc,k\}$ and $t\in [-1,1]$ such that $\{f^N(x_0),\dotsc,f^N(x_{l-1})\}\subset \U^s(\C_i)\times \{t\}$. Moreover, $\{f^N(x_0),\dotsc,f^N(x_{l-1})\}$ belongs to the same $s$-disk in the Generalized Reeb stability neighborhood $\U^s(\C_i)\times  \{t\}$. So each $s$-disk of $\U^s(\C_i)\times \{t\}$ intersects $f^N(\C_{x_0})$ in at least $l$ distinct points.

This implies that $(\C_i\times \{t\}) \cap W^s(x)=\{x\}$ for every $x\in \C_i\times \{t\}$. Otherwise, $f^N(\C_{x_0})$ would intersect some stable leaf in at least $2l$ distinct points and this cannot be since $l=\#\{\C_{x_0}\cap W^s(x_0)\}$ is maximal in $W$ and, therefore, also in $f^N(W)$.
\end{proof}

\begin{obs} For $W$ as in Proposition \ref{csleavesR} we have that the group $\Hol(W)$ is trivial.
Indeed, any closed curve in $W$ is freely homotopic to a closed curve in $\C$, that has trivial unstable holonomy (see Lemma \ref{lemma1}). 
\end{obs}

\subsection{Adapted Hiraide arguments to rule out the existence of $\R$}\label{Hirai}

This last section is dedicated to prove that the set $\R$ as described before cannot exist. The proof we give is an adaptation of the work made by Hiraide in \cite{Hir} and Bohnet in \cite{Boh}. However, the proof itself is self-contained. 

The key advantage of Hiraide's proof over Newhouse's (Anosov case, see \cite{New}) is that the former takes place in a neighbourhood of the repeller while the latter makes a more global argument. As we want to avoid dealing with the bad set, we find it more convenient to follow Hiraide's proof. It is worth mention that because of the reasons just mentioned the authors could not directly adapt Newhouse's proof.

From now on we will work with both $\R$ and its boundary $\partial \R$ in $M$. Note that, as well as $\R$, the set $\partial \R$ is non empty, closed, saturated by the center-stable foliation, has trivial transversal holonomy and the volume of its center leaves is uniformly bounded. The set $\partial \R$ has empty interior.

Let us fix an orientation of $\W^u$. 

For every $x\in \R$ we can consider $\U^s(x)\times [-1,1]$ a $(\W^{cs},\W^u)$-product neighbourhood of $x$ with $\U^s(x)$ a small center-saturated Generalized Reeb stability neighborhood of $\C_x$ (see Remark \ref{Reebproduct}). 

Let $\{V_i=\U^s(x_i)\times (-1,1)\}_{0\leq i \leq k}$ be a finite cover of $\R$. Define $\mathcal{V}=\bigcup_{0\leq i\leq k} V_i$.

We are going to see that for certain points near $\partial \R$ the center-stable leaf through this point must remain in $\mathcal{V}$ (see ``Sandwich Lemma'' \ref{Sandwich}) while it must also intersect $\B$, thus yielding a contradiction.

\begin{obs} We can assume that:
\begin{itemize}
\item $\mathcal{V}\cap \B =\emptyset$ (this is obtained by taking each $\U^s(x)\times [-1,1]$ disjoint from $\B$).
\item $\U^s(x_i)\times \{1\}\cap \partial \R=\emptyset$ (we can assume this since $\partial \R$ has empty interior).
\end{itemize}
\end{obs}

For every $i\in \{0,\dotsc,k\}$ let $0<t_i<1$ be such that $\U(x_i)\times [t_i,1]\cap \partial \R = \U(x_i)\times \{t_i\}$. Denote each $\U(x_i)\times \{t_i\}$ as $P^+_i$. Informally speaking, $P^+_i$ is the last (according to the orientation of $\W^u$) center-stable plaque of $V_i$ that is contained $\partial \R$.

%For every $i\in \{0,..,k\}$ define $P^+_i=\U(x_i)\times \{t_i\}$, for $0<t_i<1$, such that $\U(x_i)\times [t_i,1]\cap \partial \R = \U(x_i)\times \{t_i\}$. Informally speaking, $P^+_i$ is the last (according to the orientation of $\W^u$) center-stable plaque of $V_i$ that is contained $\partial \R$.

\begin{lema} There exists a pair $(x_0,\delta)\in \partial \R\times \mathbb{R}^+$ such that either $(x_0,x_0+\delta)^u\cap \R =\emptyset$ or $(x_0-\delta,x_0)^u\cap \R =\emptyset$.
\end{lema}
\begin{proof}
Observe first that the set $\R$ cannot be saturated by the unstable foliation (because in that case it would be all $M$), so there must exist a point $x \in \R$ such that $W^u(x)\cap M\setminus \R \neq \emptyset$. 

As $\R\cap W^u(x)$ is closed in $W^u(x)$ with the intrinsic topology then there must be at least one connected component $I$ of $W^u(x)\setminus \R$. Choose $x_0$ an endpoint of $I$.
\end{proof}

Fix $(x_0, \delta)$ given by the previous lemma. Assume that $(x_0,x_0+\delta)^u\cap \R =\emptyset$ (otherwise, simply change the orientation of $\W^u$).

By Proposition \ref{completeness} and Proposition \ref{csleavesR} we have $W^{cs}(x_0)=\bigcup_{x\in \C}W^s(x)$ for some center leaf $\C\subset W^{cs}(x_0)$ such that $\C\cap W^s(x)=\{x\}$ for every $x\in \C$. Assume, without loss of generality, that $\C_{x_0}$ is such a center leaf.

Some of the plaques $\{P^+_1,\dotsc,P^+_k\}$ could possibly be contained in $W^{cs}(x_0)$. Denote all of them as $\{P^+_{i_1},\dotsc,P^+_{i_m}\}$. We can consider now $N>0$ large enough as to assure that:

$$(P^+_{i_1}\cup \dotsb \cup P^+_{i_m}) \subset \bigcup_{x\in \C_{x_0}}\Ball^s_N(x).$$

If $W ^{cs}(x_0)$ contains none of the plaques $\{P^+_1,\dotsc,P^+_k\}$ then take $N>0$ any positive number. 

For simplicity, let us denote $\bigcup_{x\in \C_{x_0}}\Ball^s_N(x)$ as $E$. The subset $E$ of $W^{cs}(x_0)$ is then in the hypothesis of the Proposition \ref{productnbhd} (in particular, $\Hol^u(E)=\Id$) so we can consider a $(\W^{cs},\W^u)$-product neighborhood $E\times [-1,1]$ of it. Recall that this means that each $E\times \{y\}$ lies in a center-stable leaf, each $\{x\}\times [-1,1]$ lies in an unstable leaf and $E\times \{0\}=E$.

By eventually shriking it in the unstable direction, we can assume that $E\times [-1,1]$ is contained in $\mathcal{V}$.

\begin{lema} Through every $x\in W^{cs}(x_0)\setminus E$ there exists $L_x=[x,x+\delta_x ]^u$ a closed, non trivial unstable segment that intersects $\R$ just in its endpoints. Moreover, $L_x$ varies continuously with $x$ in $W^{cs}(x_0)\setminus E$ and is contained in any element $V_i$ of the covering that contains $x$.
\end{lema}
\begin{proof}
Let $x\in W^{cs}(x_0)\setminus E$. The point $x$ lies in some $V_i$. Then it must be that the connected component of $(W^u_+(x)\cup \{x\})\cap V_i$ that contains $x$ intersects $\R$ in at least some other point distinct from $x$ (since $x\notin \bigcup_{0\leq i\leq k}P^+_i$). 

Observe that the fact that $x$ lies inside $W ^{cs}(x_0)$ and $(x_0,x_0+\delta)^u$ is disjoint from $\R$ implies that $x$ cannot be accumulated in $W^u_+(x)\cup \{x\}$ by points of $W^u_+(x)\cap \R$. The existence of the stated $\delta_x>0$ follows.

The definition of $L_x$ does clearly not depend on the choice of the $V_i$ containing $x$. The continuous dependence on $x$ follows.
\end{proof}

By shrinking $E\times [-1,1]$ even more (if necessary) in the unstable direction, we can assume that $\{x\}\times [0,1]\subset L_x$ for every $x\in \partial E$.

Now, let $y\in (x_0,x_0+\delta)^u$ be close enough to $x_0$ such that $y\in E\times [0,1]$. Since $y\in M\setminus \R$ we have that $W^{cs}(y)\cap \B\neq \emptyset$. We are going to see that, on the other hand, $W^{cs}(y)\subset \mathcal{V}$ and this will yield a contradiction since $\mathcal{V}\cap \B=\emptyset$. 

The proof then is finished up to the following:

\begin{lema}[Sandwich lemma]\label{Sandwich} For $y$ as above we have $W^{cs}(y)\subset \mathcal{V}$.
\end{lema}
\begin{proof}

First, note that $\bigcup_{x\in W^{cs}(x_0)\setminus E}L_x$ is a foliated interval-bundle with base $W^{cs}(x_0)\setminus E$ and fibers the $L_x$'s that are transversal to $\W^{cs}$ for every $x$. 

We have then a well defined projection along fibers $\pi:\bigcup_{x\in W^{cs}(x_0)\setminus E}L_x\to W^{cs}(x_0)\setminus E$ given by $\pi([x,x+\delta_x))=x$ for $L_x=[x,x+\delta_x]$ .

For every $\gamma:[0,1]\to W^{cs}(x_0)\setminus E$ and $z\in L_{\gamma(0)}$ we can lift $\gamma$ to $\gamma_z:[0,1]\to W^{cs}(z)$ such that $\pi\circ\gamma_z=\gamma$. 

This lift defines a %holonomy map
projection $p_\gamma:L_{\gamma(0)}\to L_{\gamma(1)}$ given by $p_\gamma(z)= \gamma_z(1)$. 

Fix $x_0'\in \partial E$.

\vspace{0.2cm}
\emph{Claim.} For every curve $\gamma:[0,1]\to W^{cs}(x_0)\setminus E$ such that $\gamma(0)=\gamma(1)=x_0'$ there exists an endpoints fixed homotopy $\gamma_s:[0,1]\to W^{cs}(x_0)\setminus E$ such that $\gamma_0=\gamma$ and $\gamma_1([0,1])\subset \partial E$.

\vspace{0.2cm}
\emph{Proof of claim.}
For every $s\in (0,1]$ we can consider a retraction $r_s:W^{cs}(x_0)=\bigcup_{x\in \C_{x_0}}W^s(x)\to \bigcup_{x\in \C_{x_0}}\Ball^s_{N/s}(x)$ such that $r_s$ varies continously with $s$ and is the identity in $E=\bigcup_{x\in \C_{x_0}}\Ball^s_{N}(x)$. For every $s\in (0,1]$ compose $\gamma$ with $r_s$ to get a curve $\gamma_s$ and set $\gamma_0$ as $\gamma$. Then the homotopy $\gamma_s$ is as desired and this proves the claim.
\vspace{0.3cm}

Now, for every $x\in W^{cs}(x_0)\setminus E$ consider $\gamma_x:[0,1]\to W^{cs}(x_0)\setminus E$ such that $\gamma_x(0)=x_0'$ and $\gamma_x(1)=x$. 

\begin{figure}[htb]
  \centering
  \def\svgwidth{300pt}
  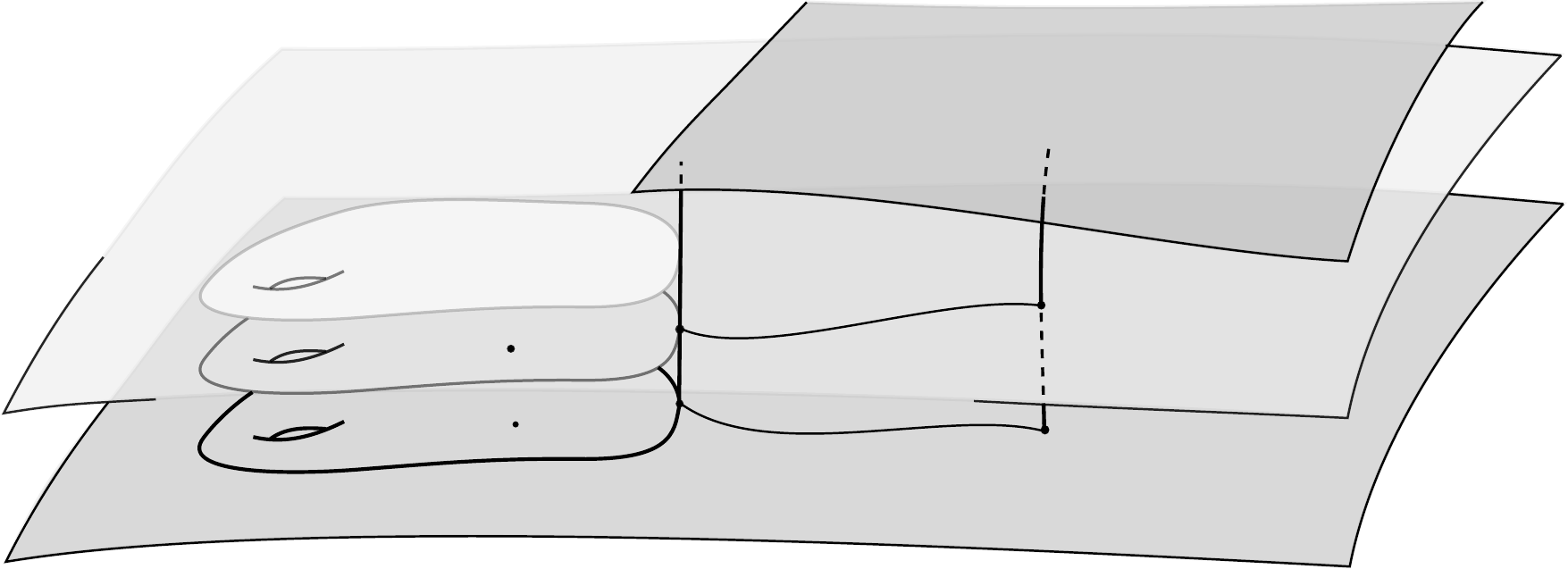
  \caption{}\label{figHir}
\end{figure}

\vspace{0.2cm}
\emph{Remark.} The existence of such a $\gamma_x$ is guaranteed if $\mathrm{dim}(E^s)\geq 2$ for this implies that $W^{cs}(x_0)\setminus E$ is path connected. Otherwise, if $\mathrm{dim}(E^s)=1$ and  $W^{cs}(x_0)\setminus E$ has two connected components simply add $x_0''\in \partial E$ not in the same connected component as $x_0'$ and reproduce the arguments that follow separately in each connected component with $x_0'$ and $x_0''$ playing the same role.
\vspace{0.2cm}

Denote $y=(x_0,t)$ in $E\times (0,1)$. Observe that we can extend $E\times (0,1)$ to a $(\W^{cs},\W^u)$-product neighborhood $(E\cup \partial E)\times (0,1)$. Denote $y'$ the point $(x_0',t)$ in this coordinates. 

Define $F:W^{cs}(x_0)\setminus E\to W^{cs}(y)$ by $$F(x)=p_{\gamma_x}(y').$$
\vspace{0.2cm}
\emph{Claim.} The definition of $F$ does not depend on the choice of $\gamma_x$.

\vspace{0.2cm}
\emph{Proof of claim.}
If $\gamma_x':[0,1]\to W^{cs}(x_0)\setminus E$ is another path such that $\gamma_x'(0)=x_0'$ and $\gamma_x'(1)=x$, we have to prove that $\gamma_x'^{-1}*\gamma_x$ lifts to a closed path $(\gamma_x*\gamma_x'^{-1})_{y'}$ from $y'$ to itself.

By the first claim, $\gamma_x'^{-1}*\gamma_x$ is homotopic to some closed path $\alpha:[0,1]\to W^{cs}(x_0)\setminus E$ such that $\alpha([0,1])\subset \partial E$. This homotopy lifts to an homotopy contained in $W^{cs}(y')$ from $(\gamma_x'^{-1}*\gamma_x)_{y'}$ to $\alpha_{y'}$, with $\alpha_{y'}$ being the lift of $\alpha$ from $y'$. 
 
Recall that we have $y'\in (E\cup\partial E)\times \{t\}$. Then the lift $\alpha_{y'}$ has to be closed because $\alpha_{y'}([0,1])$ is contained in $W^{cs}(y')$ and this implies $\alpha_{y'}([0,1])\subset (\partial E\times \{t\})$ and then $\alpha_{y'}(1)=(\partial E\times \{t\})\cap L_{x_0'}=y'$.
This proves the claim.
\vspace{0.3cm}

So we have a well defined map $F:W^{cs}(x_0)\setminus E\to W^{cs}(y)$. We can extend $F$ to $E$ in the natural way: $F(x)=(x,t)$ for every $x\in E$. Then we have $F:W^{cs}(x_0)\to W^{cs}(y)$.

The map $F$ is clearly an injective local homeomorphism. Given a Lebesgue number $\eta>0$ for the covering $\mathcal{V}=\bigcup_{0\leq i\leq k} V_i$ there exists $\delta>0$ such that for every $x\in W^{cs}(x_0)$ we have that $\Ball^{cs}_\delta(F(x))\subset F(\Ball^{cs}_{\eta/2}(x))$. We deduce from this that $F$ is a proper map. Then $F$ is also surjective and this implies that $W^{cs}(y)\subset \mathcal{V}$ since $E\times \{t\}$ and every $L_x$ are contained in $\mathcal{V}$.
\end{proof}


\begin{thebibliography}{texttLLLL}

\bibitem[B13]{Boh} D. Bohnet, Codimension-1 partially hyperbolic diffeomorphisms with a uniformly compact center foliation, \emph{Journal of Modern Dynamics}, {\bf 7} 4 (2013), 565--604. 

\bibitem[BB16]{BoBo} D. Bohnet, C. Bonatti, Partially hyperbolic diffeomorphisms with a uniformly compact center foliation: The quotient dynamics, \emph{Ergodic Theory and Dynamical Systems}, {\bf 36} 4 (2016), 1067--1105. 

\bibitem[BW05]{BoW} C. Bonatti, A. Wilkinson, Transitive partially hyperbolic diffeomorphisms on 3-manifolds, \emph{Topology}, {\bf 44} 3 (2005), 475--508.

\bibitem[CLN85]{CLN} C. Camacho, A. Lins Neto, Geometric Theory of Foliations,  \emph{Birkhäuser}, 1985. viii+206 pp.

\bibitem[CC00]{CC} A. Candel, L. Conlon, Foliations I, \emph{Graduate Studies in Mathematics}, 23, American Mathematical Society, Providence, RI, 2000. xiv+402 pp.


\bibitem[C15]{Car} P. Carrasco, Compact dynamical foliations, \emph{Ergodic Theory and Dynamical Systems}, {\bf 35} 8 (2015), 2474--2498.

\bibitem[EMS77]{EMS} R. Edwards, K. Millett, and D. Sullivan. Foliations with all leaves compact. \emph{Topology}, {\bf 16} 1 (1977), 13–32.

\bibitem[E76]{Eps} D. B. A. Epstein, Foliations with all leaves compact,  \emph{Ann. Inst. Fourier (Grenoble)}, {\bf 26} 1 (1976), 265--282.


\bibitem[EV78]{EV} D. B. A. Epstein, E. Vogt, A counterexample to the periodic orbit conjecture in codimension 3, \emph{Annals of Math.}, (2) {\bf 108} (3) (1978), 539--552.


\bibitem[G12]{Gog} A. Gogolev, Partially hyperbolic diffeomorphisms with compact center foliations, \emph{Journal of Modern Dynamics}, {\bf 5} 4 (2012), 747--769.

\bibitem[HH87]{HH} G. Hector and U. Hirsch, Introduction to the geometry of foliations, Part B, second ed., \emph{Aspects of Mathematics}, E3, Friedr. Vieweg and Sohn, Braunschweig, 1987.

\bibitem[H01]{Hir} K. Hiraide, A simple proof of the Franks-Newhouse theorem on codimension-one Anosov diffeomorphisms, \emph{Ergodic Theory and Dynamical Systems}, {\bf 21} 3 (2001), 801--806.


\bibitem[HPS77]{HPS} M. Hirsch, C. Pugh and M. Shub, Invariant Manifolds,  \emph{Springer Lecture Notes in Math.}, {\bf 583} (1977).

\bibitem[L15]{Les} P. Lessa, Reeb stability and the Gromov-Hausdorff limits of leaves in compact foliations, \emph{Asian Journal of Mathematics}, {\bf 19} 3 (2015), 433--464.

\bibitem[N70]{New} S. E. Newhouse, On Codimension One Anosov Diffeomorphisms, \emph{American Journal of Mathematics}, {\bf 92} 3 (1970), 761–770.


\bibitem[RHRHU07]{RHRHU} F. Rodriguez Hertz, M. A. Rodriguez Hertz, R. Ures, A survey of
partially hyperbolic dynamics, Partially hyperbolic dynamics,
laminations, and Teichm\"uler flow, 35--87, Fields Inst. Commun.,
51, Amer. Math. Soc., Providence, RI, 2007.


\bibitem[S76]{S} D. Sullivan, A counterexample to the periodic orbit
conjecture \emph{Inst. Hautes \'{E}tudes Sci. Publ. Math.} {\bf 46} (1976), 5--14.

\bibitem[V77]{Vog} E. Vogt, A periodic flow with infinite Epstein hierarchy \emph{Manuscripta Math.} {\bf 22} 4 (1977), 403–412.

\end{thebibliography}
\end{document}